\documentclass[10pt]{amsart}
\usepackage{tikz-cd}
\usepackage{mathrsfs}
\usepackage[shortlabels]{enumitem}
\usepackage[colorlinks,linkcolor=black,citecolor=black,urlcolor=black]{hyperref}

\numberwithin{equation}{section}
\newtheorem{thm}[subsection]{Theorem}
\newtheorem{cor}[subsection]{Corollary}
\newtheorem{lem}[subsection]{Lemma}
\newtheorem{prop}[subsection]{Proposition}
\theoremstyle{definition}
\newtheorem{df}[subsection]{Definition}
\newtheorem{rmk}[subsection]{Remark}
\newtheorem{exm}[subsection]{Example}
\newtheorem{const}[subsection]{Construction}

\newcommand{\bE}{\mathbf{E}}

\newcommand{\bM}{\mathbf{M}}

\newcommand{\A}{\mathbb{A}}

\newcommand{\E}{\mathbb{E}}
\newcommand{\F}{\mathbb{F}}

\renewcommand{\L}{\mathbb{L}}

\newcommand{\N}{\mathbb{N}}
\renewcommand{\P}{\mathbb{P}}
\newcommand{\Q}{\mathbb{Q}}

\newcommand{\Z}{\mathbb{Z}}

\newcommand{\cC}{\mathcal{C}}
\newcommand{\cD}{\mathcal{D}}

\newcommand{\cI}{\mathcal{I}}

\newcommand{\cO}{\mathcal{O}}

\newcommand{\cV}{\mathcal{V}}

\newcommand{\rD}{\mathrm{D}}

\newcommand{\rN}{\mathrm{N}}

\DeclareMathOperator{\Hom}{Hom}
\DeclareMathOperator{\Spec}{Spec}

\newcommand{\id}{\mathrm{id}}
\newcommand{\ul}{\underline}

\newcommand{\pt}{\mathrm{pt}}
\newcommand{\lSm}{\mathrm{lSm}}
\newcommand{\lSch}{\mathrm{lSch}}

\newcommand{\Bl}{\mathrm{Bl}}
\newcommand{\Gys}{\mathrm{Gys}}
\newcommand{\Tub}{\mathrm{Tub}}

\newcommand{\Fil}{\mathrm{Fil}}

\newcommand{\gr}{\mathrm{gr}}
\DeclareMathOperator{\Fun}{Fun}

\DeclareMathOperator{\cofib}{cofib}

\newcommand{\lQSyn}{\mathrm{lQSyn}}

\newcommand{\lQRSPerfd}{\mathrm{lQRSPerfd}}
\newcommand{\Poly}{\mathrm{Poly}}

\newcommand{\Sch}{\mathrm{Sch}}
\newcommand{\CycSp}{\mathrm{CycSp}}

\newcommand{\SH}{\mathrm{SH}}

\newcommand{\ex}{\mathrm{ex}}

\newcommand{\HH}{\mathrm{HH}}
\newcommand{\HC}{\mathrm{HC}}
\newcommand{\HP}{\mathrm{HP}}
\newcommand{\THH}{\mathrm{THH}}
\newcommand{\TC}{\mathrm{TC}}
\newcommand{\TP}{\mathrm{TP}}
\newcommand{\BMS}{\mathrm{BMS}}

\newcommand{\Zar}{\mathrm{Zar}}

\newcommand{\syn}{\mathrm{syn}}

\newcommand{\ket}{\mathrm{k\acute{e}t}}

\newcommand{\cdh}{\mathrm{cdh}}

\newcommand{\Th}{\mathrm{Th}}

\newcommand{\KGL}{\mathbf{KGL}}

\newcommand{\KQ}{\mathbf{KQ}}
\newcommand{\MGL}{\mathbf{MGL}}

\usepackage[bbgreekl]{mathbbol}
\usepackage{relsize}

\DeclareSymbolFontAlphabet{\mathbb}{AMSb} 
\DeclareSymbolFontAlphabet{\mathbbl}{bbold} 
\newcommand{\cPrism}{\widehat{\mathlarger{\mathbbl{\Delta}}}}

\begin{document}
\title{Logarithmic Gysin sequences for regular immersions}
\author{Doosung Park}
\address{Department of Mathematics and Informatics, University of Wuppertal, Germany}
\email{dpark@uni-wuppertal.de}
\begin{abstract}
For a regular immersion of schemes $Z\to X$ and a cohomology theory of fs log schemes, we formulate the logarithmic Gysin sequence using the ``logarithmic compactification'' $(\mathrm{Bl}_Z X,E)$ instead of the open complement $X-Z$, where $E$ is the exceptional divisor. We show that all $\mathbb{A}^1$-invariant cohomology theories produced from motivic spectra and various non $\mathbb{A}^1$-invariant cohomology theories like Nygaard completed prismatic cohomology admit logarithmic Gysin sequences.
\end{abstract}
\maketitle

\section{Introduction}

Let $\E$ be an $\A^1$-invariant Nisnevich sheaf of spectra on the category of schemes.
A fundamental theorem of Morel-Voevodsky \cite[Theorem 2.23 in \S 3]{MV} implies that if $Z\to X$ is a closed immersion of smooth schemes over a scheme,
then there is a natural fiber sequence
\[
\E(\Th(\rN_Z X))
\to
\E(X)
\to
\E(X-Z),
\]
which is called the \emph{Gysin sequence}.
Here, $\rN_Z X$ denotes the normal bundle of $Z$ in $X$,
and $\Th(\rN_Z X)$ denotes its Thom space.
It is important to have all three terms in the Gysin sequence for computational aspects.

\

If $Z\to X$ is a regular immersion of schemes,
then there is no such a Gysin sequence for general $\E$.
If we further assume that $X$ and $Z$ are regular,
then the existence of a Gysin sequence is a special property for $\E$ characterized by D\'eglise \cite[Definition 1.3.2]{MR3930052} in the motivic setting, which is called the \emph{absolute purity}.
The algebraic $K$-theory spectrum satisfies the absolute purity as a consequence of Quillen's localization sequence for algebraic $K$-theory.
We refer to \cite[Example 4.3.13]{MR4321205} for more examples satisfying the absolute purity.

\

The situation is worse for non $\A^1$-invariant cohomology theories,
e.g., Hodge cohomology, topological Hochschild homology, prismatic cohomology, syntomic cohomology, etc.
The above Gysin sequence does not exist even for the smooth case.
The non-properness of the open immersion $X-Z\to X$ causes such a failure.

\

The logarithmic approach helps in this situation.
For effective Cartier divisors $D_1,\ldots,D_n$ on $X$ given by invertible sheaves of ideals $\cI_1,\ldots,\cI_n$,
let $(X,D_1+\cdots+D_n)$ be the log scheme whose log structure is obtained by the Deligne-Faltings structure \cite[\S III.1.7]{Ogu} associated with the inclusions $\cI_1,\ldots,\cI_n\to \cO_X$.
We consider the \emph{logarithmic compactification} $(\Bl_Z X,E)$ of $X-Z$ in $X$,
where $E$ is the exceptional divisor on the blow-up $\Bl_Z X$.
The logarithmic compactification of the interval $\A^1$ in $\P^1$ is $\square:=(\P^1,\infty)$.
The \emph{logarithmic deformation space} is 
\[
\rD_Z X:=\Bl_{Z\times \{0\}}(X\times \square)-\Bl_{Z\times \{0\}}(X\times \{0\}).
\]
If we use $\A^1$ instead of $\square$ here, then we recover the classical deformation space.
We have the induced diagram of fs log schemes with cartesian squares
\[
\begin{tikzcd}
X\ar[d]\ar[r]&
\rD_Z X\ar[d]\ar[r,leftarrow]&
\rN_Z X\ar[d]
\\
\{1\}\ar[r,"i_1"]&
\square\ar[r,leftarrow,"i_0"]&
\{0\},
\end{tikzcd}
\]
where $i_0$ (resp.\ $i_1$) is the $0$-section (resp.\ $1$-section).
Observe also that $\rD_Z X$ contains $Z\times \square$ as a strict closed subscheme.
Using these replacements,
we can formulate the logarithmic Gysin sequences as follows:

\begin{df}
\label{intro.3}
Let $\E$ be a presheaf on the category of fs log schemes with values in a stable $\infty$-category $\cV$.
We say that $\E$ admits \emph{logarithmic Gysin sequences} if the squares in the induced commutative diagram
\begin{equation}
\label{intro.1.1}
\begin{tikzcd}
\E(X)\ar[d]\ar[r,leftarrow]&
\E(\rD_Z X)\ar[d]\ar[r]&
\E(\rN_Z X)\ar[d]
\\
\E(\Bl_Z X,E)\ar[r,leftarrow]&
\E(\Bl_{Z\times \square}(\rD_Z X),E^D)\ar[r]&
\E(\Bl_Z (\rN_Z X),E^N)
\end{tikzcd}
\end{equation}
are cartesian for every regular immersion of schemes $Z\to X$,
where $E$, $E^D$, and $E^N$ are the exceptional divisors.
In this case,
the induced fiber sequence
\[
\E(\Th(\rN_Z X))
\to
\E(X)
\to
\E(\Bl_Z X)
\]
is called the \emph{logarithmic Gysin sequence},
where
\[
\E(\Th(\rN_Z X))
:=
\cofib(\E(\rN_Z X)\to \E(\Bl_Z (\rN_Z X),E^N)).
\]
By \cite[Proposition 7.4.5]{BPO},
there is a natural isomorphism
\[
\E(\Th(\rN_Z X))
\simeq
\cofib(\E(\P(\rN_Z X\oplus \cO))\to \E(\P(\rN_Z X))).
\]
Hence the \emph{Gysin morphism} $\E(\Th(\rN_Z X))\to \E(X)$ can be written purely in terms of schemes.
\end{df}

The author's joint work with Binda and {\O}stv{\ae}r \cite[\S 7]{BPO} proved that $(\P^\bullet,\P^{\bullet-1})$-invariant strict Nisnevich sheaves admit logarithmic Gysin sequences for the smooth case,
i.e., when $Z\to X$ is a closed immersion of smooth schemes over a scheme $S$.
This is a generalization of \cite[Theorem 2.23 in \S 3]{MV} to the non $\A^1$-invariant setting.
We refer to \cite{BPO2}, \cite{BLPO}, and \cite{BLMP} for the applications to various non $\A^1$-invariant cohomology theories.

\

The purpose of this paper is to show that fairly large classes of cohomology theories admit logarithmic Gysin sequences as follows:

\begin{thm}
\label{intro.2}
The following presheaves $X\mapsto \E(X)$ on the category of fs log schemes admit logarithmic Gysin sequences:
\begin{enumerate}
\item[\textup{(1)}] There exists an object $\bE\in \SH(\Z)$ such that
\[
\E(X)=
\hom_{\SH(X)}(\Sigma^\infty X_+,p^*\bE),
\]
where $p\colon X\to \Spec(\Z)$ is the structure morphism,
\item[\textup{(2)}] $\E(X)=R\Gamma_{\ket}(X,\Z/n)$ for every integer $n>1$,
\item[\textup{(3)}] $\E(X)=R\Gamma_{\Zar}(X,\L\Omega_{X/S}^j)$ for every integer $j$,
\item[\textup{(4)}] $\E(X)=\HH(X),\HC^-(X),\HP(X)$.
\item[\textup{(5)}] $\E(X)=\THH(X),\TC^-(X),\TP(X),\TC(X)$,
\item[\textup{(6)}] $\E(X)=R\Gamma_{\cPrism}(X),R\Gamma_{\syn}(X,\Z_p(d))$ for every integer $d$ with noetherian $X$ such that the $p$-completion $X_p^\wedge$ is log quasi-syntomic.
\end{enumerate}
\end{thm}
\begin{proof}
See Theorem \ref{cdh.6} for (1),
Theorem \ref{cdh.8} for (2),
Corollary \ref{HH.4} for (3),
Theorem \ref{HH.3} and Proposition \ref{HH.5} for (4),
Theorem \ref{THH.5} and Corollary \ref{THH.6} for (5),
and Theorem \ref{THH.3} and Corollary \ref{THH.4} for (6).
\end{proof}

The proof of (1) relies on \cite[Theorem 3.6.7]{logshriek},
whose proof is complicated.
On the other hand,
by formal arguments,
we reduce (3)--(6) to the known smooth cases dealt in \cite{BPO2}, \cite{BLPO}, and \cite{BLMP}.

\

In the case of (1),
we expect that the Gysin morphism
\[
\E(\Th(\rN_Z X))
\to
\E(X)
\]
for regular immersion of schemes $Z\to X$ agrees with the Gysin morphism constructed by D\'eglise-Jin-Khan \cite[4.3.3]{MR4321205}.
This holds at least in the smooth case since in this case,
\cite[Lemma 3.2.15]{MR4321205} implies that their Gysin morphism agrees with the Gysin morphism of Morel-Voevodsky.

\

In \S 3,
we recall the logarithmic absolute purity property,
and we show that this is equivalent to the absolute purity property under certain conditions.
As a consequence,
we prove the following result about the log nearby cycles functor that will be studied further in \cite{lognearby}.
We refer to \cite[Definition 2.3.1]{logshriek} for $\SH^\ex$.

\begin{thm}[See Theorem \ref{absolute.6}]
Let $X$ be a regular log regular log scheme over $B$,
let $i\colon \partial X\to X$ and $j\colon X-\partial X\to X$ be the obvious strict immersions,
and let $\bE$ be an object of $\SH(B)$.
If $\bE$ satisfies the absolute purity,
then $\bE$ satisfies the logarithmic absolute purity in the sense that 
there is a natural isomorphism in $\SH^{\ex}(X)$
\[
\bE\simeq j_*\bE,
\]
where we set $\bE:=p^*\bE$ for every morphism of fs log schemes $p$ whose target is $B$.
Hence there is a natural isomorphism in $\SH^{\ex}(\partial X)$
\[
\bE\simeq \Psi^{\log}\bE,
\]
where $\Psi^{\log}:=i^*j_*$ is the log nearby cycles functor.
\end{thm}

In particular,
if $X$ is the fs log scheme $(\Spec(\Z_p),\Spec(\F_p))$,
then the above functor $\Psi^{\log}$ creates the homotopy $K$-theory spectra of log smooth schemes over the standard log point $\pt_{\N,\F_p}$ from the homotopy $K$-theory spectra of smooth schemes over $\Spec(\Q_p)$.

\subsection*{Acknowledgements}
The author wishes to thank Adeel Khan for a helpful comment on this paper.
This research was conducted in the framework of the DFG-funded research training group GRK 2240: \emph{Algebro-Geometric Methods in Algebra, Arithmetic and Topology}.

\subsection*{Notation and conventions}

Our standard reference for log geometry is Ogus's book \cite{Ogu}.
We assume that every fs log scheme in this paper has a Zariski log structure unless otherwise stated.
We employ the following notation throughout this paper:

\begin{tabular}{l|l}
$B$ & a base scheme
\\
$\Sch$ & the category of schemes
\\
$\lSch$ &  the category of fs log schemes
\\
$\Hom_{\cC}$ & the hom space in an $\infty$-category $\cC$
\\
$\hom_{\cC}$ & the hom spectra in a stable $\infty$-category $\cC$
\\
$\cV$ & a stable $\infty$-category
\end{tabular}

\section{Log cdh sheaves}

In this section,
we show that $(\P^\bullet,\P^{\bullet-1})$-invariant log cdh sheaves admit logarithmic Gysin sequences.
The cohomology theory of fs log schemes produced by a motivic spectrum in the sense of Theorem \ref{cdh.6} is such an example.
Kummer \'etale cohomology $R\Gamma_{\ket}(X,\Z/n)$ is another example, which is not $\A^1$-invariant if $n$ is not invertible in $X$.

A presheaf $\E$ on $\lSch$ with values in $\cV$ is \emph{$(\P^\bullet,\P^{\bullet-1})$-invariant} if the induced morphism
\[
\E(X)
\to
\E(X\times (\P^n,\P^{n-1}))
\]
is an isomorphism for $X\in \lSch$ and integer $n\geq 1$,
where we regard $\P^{n-1}$ as the divisor on $\P^n$ at $\infty$.
Similarly, $\E$ is \emph{$\square$-invariant} if the induced morphism
\[
\E(X)
\to
\E(X\times \square)
\]
is an isomorphism for $X\in \lSch$.

We refer to \cite[Definition 2.3.13]{BPO2} for the \emph{strict Nisnevich cd-structure}, \emph{dividing Nisnevich cd-structure}, \emph{strict Nisnevich topology}, and \emph{dividing Nisnevich topology}.

\begin{df}
\label{cdh.1}
A \emph{proper log cdh distinguished square} is a cartesian square of fs log schemes
\[
Q:=
\begin{tikzcd}
W\ar[d]\ar[r]&
Y\ar[d,"f"]
\\
Z\ar[r,"i"]&
X
\end{tikzcd}
\]
such that $i$ is a strict closed immersion, $f$ is proper (i.e., the underlying morphism of schemes $\ul{f}$ is proper), and the induced morphism $f^{-1}(X-Z)\to X-Z$ is an isomorphism.

The \emph{log cdh cd-structure} is the collection of proper log cdh and strict Nisnevich distinguished squares.

The \emph{log cdh topology} is the smallest topology finer than the Zariski topology and the topology associated with the log cdh cd-structure.
\end{df}

\begin{rmk}
\label{cdh.2}
The log cdh cd-structure is complete and regular in the sense of \cite[\S 2]{Vcdtop}.
Hence a Zariski sheaf $\E$ on the category of fs log schemes with values in $\cV$ is a log cdh sheaf (i.e., $\E$ satisfies the \v{C}ech descent for the log cdh topology) if and only if $\E(Q)$ is cartesian for all proper log cdh and strict Nisnevich distinguished squares using the argument in \cite[\S 5]{Vcdtop}, see also \cite[Theorem A.4.2]{logGysin}.
\end{rmk}

\begin{exm}
\label{cdh.3}
The cartesian square
\[
\begin{tikzcd}
\pt_\N\ar[d]\ar[r]&
\A_\N\ar[d]
\\
\{0\}\ar[r,"i_0"]&
\A^1
\end{tikzcd}
\]
is a proper log cdh distinguished square,
where $i_0$ is the $0$-section.
\end{exm}

\begin{df}
Let $Z\to X$ be a regular immersion of schemes.
The \emph{log tubular neighborhood of $Z$ in $X$} is
\[
\Tub_Z X
:=
(\Bl_Z X,E)\times_{\Bl_Z X} E,
\]
where $E$ is the exceptional divisor on $\Bl_Z X$.
\end{df}

\begin{exm}
\label{cdh.4}
Let us investigate the log structure on the log tubular neighborhood in the affine case as follows.
Let $A$ be a commutative ring with a regular sequence $a_1,\ldots,a_n$.
Consider the regular immersion of schemes $Z:=\Spec(A/(a_1,\ldots,a_n))\to X:=\Spec(A)$.
We have the Zariski covering $\{U_1,\ldots,U_n\}$ of $\Bl_Z X$ with
\[
U_i:=\Spec{B_i},
\text{ }
B_i:=A[a_1/a_i,\ldots,a_n/a_i]
\]
for $i=1,\ldots,n$.
For every nonempty subset $I$ of $\{1,\ldots,n\}$,
we set $U_I:=\bigcap_{i\in I} U_i$,
which is naturally isomorphic to $\Spec(B_I)$ with
\[
B_I
:=
A[(a_j/a_i)_{i\in I,1\leq j\leq n}].
\]

Let $J_I$ be the ideal of $B_I$ defining $E\times_{\Bl_Z X}U_I$,
and let $Q_I$ be the set of pairs $(x,n)$ such that $n\in \N$ and $x\in J_I^n$.
Observe that $Q_I$ has a natural monoid structure.
We have the map $Q_I\to B_I$ sending $(x,n)$ to $x$.
Using the description of the log structures associated with Deligne-Faltings structures in \cite[\S III.1.7]{Ogu},
we have an isomorphism
\[
(\Bl_Z X,E)\times_{\Bl_Z X}U_I
\simeq
\Spec(B_I,Q_I)
\]
that is natural in $I$. 

Let $P_I$ be the submonoid of $\Z^n=\Z e_1\oplus \cdots \oplus \Z e_n$ generated by $e_i$ and $e_i-e_j$ for $i,j\in I$.
We have the natural map $P_I\to B_I$ sending $e_i$ to $a_i$ and $e_i-e_j$ to $a_i/a_j$.
Furthermore,
$P_I$ is a chart of $\Spec(B_I,Q_I)$.
Hence we have a natural isomorphism
\[
(\Bl_Z X,E)\times_{\Bl_Z X}U_I
\simeq
\Spec(B_I,P_I).
\]
In particular,
the log structure on $(\Bl_Z X,E)$ is fine saturated.

Now, we set $C:=A/(a_1,\ldots,a_n)$ and
\[
C_I:=C[(t_i)_{i\in I},(t_j/t_i)_{i\in I,1\leq j\leq n}],
\]
where $t_1,\ldots,t_n$ are indeterminates.
Observe that $C_I$ is naturally isomorphic to $B_I/(a_1,\ldots,a_n)$.
The composite map $P_I\to C_I$ sends $e_i$ to $0$ and $e_i-e_j$ to $t_i/t_j$.
Furthermore,
we have a natural isomorphism
\[
\Tub_Z X\times_{\Bl_Z X} U_I
\simeq
\Spec(C_I,P_I).
\]

This implies that the log structure on $\Tub_Z X$ is uniquely determined by $Z$ and $n$ and independent of the choice of $X$.

Note that this claim is only for the affine case.
If $Z$ is not affine,
then $\Tub_Z X$ is not uniquely determined by $Z$ and $n$.
\end{exm}

Recall the following fundamental result on the theme of logarithmic Gysin sequences:

\begin{thm}
\label{cdh.7}
Let $\E$ be a presheaf on $\lSch$ with values in $\cV$.
Then $\E$ is a $(\P^\bullet,\P^{\bullet-1})$-invariant strict Nisnevich sheaf if and only if $\E$ is a $\square$-invariant dividing Nisnevich sheaf.
In this case,
for every closed immersion of smooth schemes $Z\to X$ over a scheme $S$,
the squares in \eqref{intro.1.1} are cartesian.
\end{thm}
\begin{proof}
For the noetherian case,
this is a direct consequence of \cite[Proposition 7.3.1, Theorems 7.5.4, 7.7.4]{BPO}.
As in \cite[Theorems 3.2.21, 3.5.6, 3.5.7]{BPO2},
use the noetherian reduction argument in  \cite[\S 3.1]{BPO2}.
\end{proof}

With log cdh descent that is stronger than strict Nisnevich descent,
we have logarithmic Gysin sequences for regular immersions as follows:

\begin{thm}
\label{cdh.5}
Let $\E$ be a $(\P^\bullet,\P^{\bullet-1})$-invariant log cdh sheaf on $\lSch$ with values in $\cV$.
Then $\E$ admits logarithmic Gysin sequences.
\end{thm}
\begin{proof}
Let $Z\to X$ be a regular immersion of schemes.
We will only focus on showing that the left square of \eqref{intro.1.1} for $Z\to X$ is cartesian since the proof for the right square is similar.
Since $\E$ is a strict Nisnevich sheaf,
the question is Zariski local on $X$.
Hence we may assume that $X$ is an affine regular scheme $\Spec(A)$ and $Z$ is given by a regular sequence $a_1,\ldots,a_n$ in $A$.
Consider
\[
\Tub_{Z\times \square}(\rD_Z X)
:=
(\Bl_{Z\times \square} (\rD_Z X),E^D)\times_{\Bl_{Z\times \square} (\rD_Z X)} E^D,
\]
where $E^D$ is the exceptional divisor.
The squares
\[
\begin{tikzcd}
\Tub_Z X\ar[d]\ar[r]&
(\Bl_Z X,E)\ar[d]
\\
Z\ar[r]&
X,
\end{tikzcd}
\text{ }
\begin{tikzcd}
\Tub_{Z\times \square}(\rD_Z X)\ar[d]\ar[r]&
(\Bl_{Z\times \square}(\rD_Z X),E^D)\ar[d]
\\
Z\times \square\ar[r]&
\rD_Z X,
\end{tikzcd}
\]
are log cdh distinguished squares.
Since $\E$ is a log cdh sheaf,
the squares
\[
\begin{tikzcd}
\E(X)\ar[d]\ar[r]&
\E(\Bl_Z X,E)\ar[d]
\\
\E(Z)\ar[r]&
\E(\Tub_Z X),
\end{tikzcd}
\text{ }
\begin{tikzcd}
\E(\rD_Z X)\ar[d]\ar[r]&
\E(\Bl_{Z\times \square}(\rD_Z X),E^D)\ar[d]
\\
\E(Z\times \square)\ar[r]&
\E(\Tub_{Z\times \square}(\rD_Z X))
\end{tikzcd}
\]
are cartesian.
Hence the claim that the left square of \eqref{intro.1.1} is cartesian is equivalent to the claim that the square
\[
\begin{tikzcd}
\E(Z)\ar[d]\ar[r]&
\E(Z\times \square)\ar[d]
\\
\E(\Tub_Z X)\ar[r]&
\E(\Tub_{Z\times \square}(\rD_Z X))
\end{tikzcd}
\]
is cartesian,
which is equivalent to the claim that $\E(\Tub_Z X)\to \E(\Tub_{Z\times \square}(\rD_Z X))$ is an isomorphism by $(\P^\bullet,\P^{\bullet-1})$-invariance.

Example \ref{cdh.4} shows that the log structure on $\Tub_Z X$ is uniquely  determined by $Z$ and $n$.
Similarly,
the log structure on $\Tub_{Z\times \square}(\rD_Z X)$ is uniquely determined by $Z$ and $n$.
Hence to show that $\E(\Tub_Z X)\to \E(\Tub_{Z\times \square}(\rD_Z X))$ is an isomorphism or equivalently the left square of \eqref{intro.1.1} is cartesian,
we may assume that $Z\to X$ is given by the $0$-section $Z\to Z\times \A^n$.
This smooth case is done by Theorem \ref{cdh.7}.
\end{proof}

\begin{rmk}
\label{cdh.9}
Let $\cC$ be a full subcategory of $\lSch$,
and let $Z\to X$ be a regular immersion of schemes in $\cC$.
Consider the commutative diagram of fs log schemes
\[
\Gys(Z\to X)
:=
\begin{tikzcd}
(\Bl_Z X,E)\ar[d]\ar[r]&
(\Bl_{Z\times \square}(\rD_Z X),E^D)\ar[r,leftarrow]\ar[d]&
(\Bl_Z(\rN_Z X),E^N)\ar[d]
\\
X\ar[r]&
\rD_Z X\ar[r,leftarrow]&
\rN_Z X.
\end{tikzcd}
\]
If every fs log schemes in $\Gys(Z\to X)$ belongs to $\cC$ for every regular immersion of schemes $Z\to X$ in $\cC$,
then we can discuss logarithmic Gysin sequences within $\cC$ as in Definition \ref{intro.3}.
A typical example of $\cC$ we will use is $\lSch/B$,
where $B$ is a base scheme.
\end{rmk}

We refer to \cite[Definition 2.5.5]{logA1} for the definition of $\SH(X)$ for fs log scheme $X$,
which is an extension of the original $\SH$ for schemes due to Morel-Voevodsky \cite{MV}.

\begin{thm}
\label{cdh.6}
Let $B$ be a scheme.
For every object $\bE$ of $\SH(B)$,
the presheaf $\E$ of spectra on $\lSch/B$ given by
\[
X\mapsto \E(X):=\hom_{\SH(X)}(\Sigma^\infty X_+,p^*\bE)
\]
is a $(\P^\bullet,\P^{\bullet-1})$-invariant log cdh sheaf,
where $p\colon X\to B$ is the structure morphism.
Hence it admits logarithmic Gysin sequences.
\end{thm}
\begin{proof}
By Theorem \ref{cdh.5}, it suffices to show the first claim.
Consider the full subcategory $\SH^\ex(X)$ of $\SH(X)$ in \cite[Definition 2.3.1]{logshriek}.
Since $\SH^\ex(B)=\SH(B)$ and $\Sigma^\infty X_+\in \SH^\ex(X)$,
we have
\[
\E(X)\simeq \hom_{\SH^\ex(X)}(\Sigma^\infty X_+,p^*\bE).
\]
Together with \cite[Corollary 3.6.8]{logshriek},
we deduce that $\E$ is a log cdh sheaf.

Since $\E$ is $\A^1$-invariant,
($ver$-inv) in \cite[Theorem 2.3.5]{logshriek} implies that $\E$ is $(\P^\bullet,\P^{\bullet-1})$-invariant.
\end{proof}

The next result shows that there exists a non $\A^1$-invariant cohomology theory to which Theorem \ref{cdh.5} is applicable.

\begin{thm}
\label{cdh.8}
For every integer $n>1$,
the presheaf of complexes $R\Gamma_{\ket}(-,\Z/n)$ on $\lSch$ is a $(\P^\bullet,\P^{\bullet-1})$-invariant log cdh sheaf.
Hence it admits logarithmic Gysin sequences.
\end{thm}
\begin{proof}
By Theorem \ref{cdh.5}, it suffices to show the first claim.
For an fs log scheme $X$,
consider the presheaf of complexes on $\lSm/\ul{X}$ given by
\[
Y
\mapsto
R\Gamma_{\ket}(Y\times_{\ul{X}}X,\Z/n).
\]
This is a dividing Nisnevich sheaf by \cite[Proposition 5.4(2)]{MR3658728} and $\square$-invariant by Theorem \ref{cdh.7},
so it is $(\P^\bullet,\P^{\bullet-1})$-invariant by \cite[Proposition 7.3.1]{BPO}.
Hence $R\Gamma_{\ket}(-,\Z/n)$ is $(\P^\bullet,\P^{\bullet-1})$-invariant too.

To show that $R\Gamma_{\ket}(-,\Z/n)$ is a log cdh sheaf,
consider the presheaf of $\infty$-categories on $\lSch$ assigning the small Kummer \'etale site $X_{\ket}$ for $X\in \lSch$ and the functor $f^*\colon X_{\ket} \to Y_{\ket}$ for a morphism $f\colon Y\to X$ in $\lSch$.
Consider a commutative diagram consisting of cartesian squares in $\lSch$
\[
\begin{tikzcd}
Z'\ar[d,"q"']\ar[r,"i'"]&
X'\ar[d,"p"]\ar[r,leftarrow,"j'"]&
U\ar[d,"\id"]
\\
Z\ar[r,"i"]&
X\ar[r,leftarrow,"j"]&
U
\end{tikzcd}
\]
such that the left square is a log cdh distinguished square and $j$ is the open complement of a strict closed immersion $i$.
We need to show that the induced square
\[
\begin{tikzcd}
\id\ar[d]\ar[r]&
i_*i^*\ar[d]
\\
p_*p^*\ar[r]&
r_*r^*
\end{tikzcd}
\]
with $r:=iq$ is cartesian.
The fibers of the rows are $j_! j^*$ and $p_*j_!'j'^*p^*$ by the localization property \cite[\S 2.8]{MR1457738}.
To conclude,
observe that we have $j_!\simeq p_*j_!'$ by \cite[\S 5.4]{MR1457738}.
\end{proof}

\begin{rmk}
Unlike $\A^1$-invariant cohomology theories,
there are lots of non $\A^1$-invariant cohomology theories of schemes that do not satisfy cdh descent, e.g., $\THH$.
This is the reason why we deal with $\THH$ and its variants separately in \S \ref{HH} and \ref{THH} below.
On the other hand,
the cdh sheafification $L_\cdh \THH$ is still an interesting non $\A^1$-invariant cohomology theory.
We do not know the existence of a $(\P^\bullet,\P^{\bullet-1})$-invariant log cdh sheaf of complexes on $\lSch$ extending $L_\cdh \THH$ from schemes to fs log schemes.
\end{rmk}

\section{Logarithmic absolute purity}

The purpose of this section is to recall the logarithmic absolute purity property and show that it is equivalent to the absolute purity property under certain conditions.

For an fs log scheme $X$,
let $\partial X$ denote the set of points $x$ of $X$ such that the log structure on $X$ is nontrivial at $x$.
We regard $\partial X$ as a strict closed subscheme of $X$ with the reduced scheme structure.

We say that an fs log scheme $X$ is \emph{regular log regular} if its underlying scheme $\ul{X}$ is regular and $X$ is log regular.

Recall that $B$ is a base scheme.

\begin{df}
\label{absolute.1}
A presheaf $\E$ on $\lSch/B$ with values in $\cV$ satisfies the \emph{logarithmic absolute purity} if for every regular log regular $X\in \lSch/B$,
the induced morphism
\[
\E(X)
\to
\E(X-\partial X)
\]
is an isomorphism.
\end{df}

\begin{rmk}
If a presheaf $\E$ on $\lSch/B$ satisfies $(\P^\bullet,\P^{\bullet-1})$-invariance and the logarithmic absolute purity,
then $\E$ is $\A^1$-invariant for regular log regular $X\in \lSch/B$:
The morphism
\[
p^*\colon \E(X)\to \E(X\times \A^1)
\]
induced by the projection $p\colon X\times \A^1\to X$ is an isomorphism.
\end{rmk}

\begin{prop}
\label{absolute.3}
A $(\P^\bullet,\P^{\bullet-1})$-invariant log cdh sheaf $\E$ on $\lSch/B$ with values in $\cV$ satisfies the logarithmic absolute purity if and only if for every codimension $1$ regular immersion of regular schemes $W\to Y$,
the induced morphism
\[
\E(Y,W)
\to
\E(Y-W)
\]
is an isomorphism.
\end{prop}
\begin{proof}
The only if direction is trivial.
To show the if direction,
for a regular log regular $X\in \lSch/B$,
we need to show that the induced morphism $\E(X)\to \E(X-\partial X)$ is an isomorphism.
Working Zariski locally on $X$,
we may assume that $X$ has a chart $\N^n$ for some integer $n\geq 0$ by \cite[Theorem III.1.11.8(2)]{Ogu}.
We proceed by induction on $n$.
The claim is trivial if $n=0,1$.
Hence assume $n\geq 2$.

Let $e_1,\ldots,e_n$ be the standard coordinates in $\N^n$,
and let $D_1,\ldots,D_n$ be the corresponding effective Cartier divisors on $\ul{X}$.
Then we have $X\simeq (\ul{X},D_1+\cdots+D_n)$.
We set
\[
Z:=
(D_1\cap \cdots \cap D_n)\times_{\ul{X}}X,
\text{ }
X':=(\Bl_{\ul{Z}}\ul{X},E+\widetilde{D}_1+\cdots+\widetilde{D}_n),
\]
where $\widetilde{D}_i$ is the strict transform of $D_i$ for $1\leq i\leq n$,
and $E$ is the exceptional divisor.
For every nonempty subset $I:=\{i_1,\ldots,i_r\}$ of $\{1,\ldots,n\}$,
consider
\[
X_I:=(\ul{X},D_{i_1}+\cdots+D_{i_r}),
\text{ }
X_I':=(\ul{Y},E+\widetilde{D}_{i_1}+\cdots+\widetilde{D}_{i_r}).
\]
As in the proof of \cite[Proposition 3.2.5]{BPO2},
we have
\[
\lim_I \E(X_I')
\simeq
\E(\Bl_{\ul{Z}} \ul{X},E)
\]
using the fact that $\E$ is a Zariski sheaf,
where the limit runs over the category associated with the partially ordered set of nonempty subsets of $\{1,\ldots,n\}$.
We will show below that the induced morphism $\E(X_I)\to \E(X_I')$ is an isomorphism for every $I$.
Assuming this,
we have
\[
\lim_I \E(X_I)
\simeq
\E(\Bl_{\ul{Z}} \ul{X},E).
\]
By induction,
we have $\E(X_I)\simeq \E(X_I-\partial X_I)$ if $I\neq \{1,\ldots,n\}$ and $\E(\Bl_{\ul{Z}} \ul{X},E)\simeq \E(\ul{X}-\ul{Z})$.
Since $\E$ is a Zariski sheaf,
we have
\[
\lim_I \E(X_I-\partial X_I)
\simeq
\E(\ul{X}-\ul{Z}).
\]
Combine the above equations to have $\E(X_I)\simeq \E(X_I-\partial X_I)$ if $I=\{1,\ldots,n\}$,
i.e., $\E(X)\simeq \E(X-\partial X)$.

It remains to show that $\E(X_I)\to \E(X_I')$ is an isomorphism for every $I$.
Since $\E$ is a log cdh sheaf,
the induced square
\[
\begin{tikzcd}
\E(X_I)\ar[d]\ar[r]&
\E(Z_I)\ar[d]
\\
\E(X_I')\ar[r]&
\E(Z_I')
\end{tikzcd}
\]
is cartesian,
where $Z_I:=X_I\times_{\ul{X}}\ul{Z}$ and $Z_I':=Z_I\times_{X_I} X_I'$.
Hence the claim that $\E(X)\to \E(X')$ is an isomorphism is equivalent to the claim that $\E(Z_I)\to \E(Z_I')$ is an isomorphism.

Consider $X_0:=\A_{\N^n}$,
and produce the fs log schemes $X_{0,I}$, $X_{0,I}'$, $Z_{0,I}$, and $Z_{0,I}'$ as $X$ has produced $X_I'$, $Z_I$, and $Z_I'$.
Then we have an isomorphism $Z_I'\simeq \ul{Z_I'}\times_{\ul{Z_{0,I}'}}Z_{0,I}'$.
Hence the log structure on $Z_I'$ is uniquely determined by $Z$, $I$, and $n$,
so we may assume that $X=X_0$.
This smooth case is done by \cite[Proposition 2.4.3]{logA1} since $\E$ is $\A^1$-invariant.
\end{proof}

For a regular immersion $Z\to X$ in $\lSch/B$,
consider $\rD_Z^{\A^1}X:=\rD_Z X - \partial \rD_Z X$,
which is the usual deformation space in $\A^1$-homotopy theory.

\begin{thm}
\label{absolute.4}
Let $\E$ a $(\P^\bullet,\P^{\bullet-1})$-invariant and $\A^1$-invariant log cdh sheaf on $\lSch/B$ with values in $\cV$.
Then the following are equivalent:
\begin{enumerate}
\item[\textup{(1)}] $\E$ satisfies the logarithmic absolute purity.
\item[\textup{(2)}] $\E$ satisfies the absolute purity in the sense of \cite[Definition 1.3.2]{MR3930052}, i.e., the squares in the induced diagram
\begin{equation}
\label{absolute.4.1}
\begin{tikzcd}
\E(X)\ar[d]\ar[r,leftarrow]\ar[d]&
\E(\rD_Z^{\A^1}X)\ar[r]\ar[d]&
\E(\rN_Z X)\ar[d]
\\
\E(X-Z)\ar[r,leftarrow]&
\E(\rD_Z^{\A^1} X-Z\times \A^1)\ar[r]&
\E(\rN_Z X-Z)
\end{tikzcd}
\end{equation}
are cartesian for every regular immersion $Z\to X$ in $\Sch/B$.
\item[\textup{(3)}]
The squares in \eqref{absolute.4.1} are cartesian for every codimension $1$ regular immersion $Z\to X$ in $\Sch/B$.
\end{enumerate}
\end{thm}
\begin{proof}
Let $Z\to X$ be a regular immersion (resp.\ codimension $1$ regular immersion) in $\Sch/B$.
We know that the squares in \eqref{intro.1.1} are cartesian by Theorem \ref{cdh.5}.
Hence (1) implies (2) (resp.\ (3)).

Conversely,
assume (2) (resp.\ (3)).
Then the squares in the induced commutative diagram
\[
\begin{tikzcd}
\E(\Bl_Z X,E)\ar[d]\ar[r,leftarrow]\ar[d]&
\E(\Bl_{Z\times \square}(\rD_Z X),E^D)\ar[r]\ar[d]&
\E(\Bl_Z(\rN_Z X),E^N)\ar[d]
\\
\E(X-Z)\ar[r,leftarrow]&
\E(\rD_Z^{\A^1} X-Z\times \A^1)\ar[r]&
\E(\rN_Z X-Z)
\end{tikzcd}
\]
are cartesian.
Since the $0$-section $Z\to \rN_Z X$ is a closed immersion of smooth schemes over $Z$,
the morphism $\E(\Bl_Z(\rN_Z X),E^N)\to \E(\rN_Z X-Z)$ is an isomorphism by \cite[Proposition 2.4.3]{logA1} and Theorem \ref{cdh.7}.
It follows that the morphism $\E(\Bl_Z X,E)\to \E(X-Z)$ is an isomorphism too.
Together with Proposition \ref{absolute.3},
we deduce (1) from (2) (resp.\ (3)).
\end{proof}

\begin{df}
An object $\bE\in \SH(B)$ satisfies the \emph{absolute purity} if the presheaf on $\Sch/B$ given by
\[
X\mapsto \hom_{\SH(X)}(\Sigma^{0,n}\Sigma^\infty X_+,p^*\bE)
\]
satisfies the absolute purity for every integer $n$,
where $p\colon X\to B$ is the structure morphism.
\end{df}

\begin{exm}
\label{absolute.5}
Fujiwara and Kato \cite[Theorem 7.4]{MR1922832} showed that
$R\Gamma_{\ket}(-,\Z/n)$ satisfies the logarithmic absolute purity if $n$ is invertible in $B$.
This is a consequence of Gabber's absolute purity theorem \cite{MR1971516}.

According to \cite[Example 4.3.13]{MR4321205},
the algebraic $K$-theory spectrum $\KGL$,
the hermitian $K$-theory spectrum $\KQ$ over $\Spec(\Z[1/2])$,
the rational Eilenberg-MacLane spectrum $\bM \Q$,
and the rational algebraic cobordism spectrum $\MGL \otimes \Q$ satisfy the absolute purity.
\end{exm}

\begin{thm}
\label{absolute.6}
Let $X$ be a regular log regular log scheme over $B$,
let $i\colon \partial X\to X$ and $j\colon X-\partial X\to X$ be the obvious strict immersions,
and let $\bE$ be an object of $\SH(B)$.
If $\bE$ satisfies the absolute purity,
then $\bE$ satisfies the logarithmic absolute purity in the sense that 
there is a natural isomorphism in $\SH^{\ex}(X)$
\[
\bE\simeq j_*\bE,
\]
where we set $\bE:=p^*\bE$ for every morphism of fs log schemes $p$ whose target is $B$.
Hence there is a natural isomorphism in $\SH^{\ex}(\partial X)$
\[
\bE\simeq \Psi^{\log}\bE,
\]
where $\Psi^{\log}:=i^*j_*$ is the log nearby cycles functor.
\end{thm}
\begin{proof}
This is an immediate consequence of Theorem \ref{absolute.4}.
\end{proof}

\section{Hochschild homology}
\label{HH}

Let $R$ be a commutative ring throughout this section.

A log ring (often called pre-log ring) is a map of commutative monoids $M\to A$,
where $A$ is a commutative ring whose monoid structure is the multiplication.
A commutative ring $A$ can be considered as the log ring $(A,\{1\})$.

For every log $R$-algebra $(A,M)$,
we have Hochschild homology $\HH((A,M)/R)$,
see \cite[Definition 5.3]{BLPO}.
We can globalize this to obtain Hochschild homology $\HH(X/R)$ for log scheme $X$ over $R$,
see \cite[Definition 7.15]{BLPO} for the noetherian case and \cite[Definition 8.3.7, Proposition 8.3.8]{BPO2} to avoid the noetherian assumption.
We refer to \cite[\S 2.11]{BLPO2} for a natural $S^1$-action on $\HH((A,M)/R)$.
From this,
we also have a natural $S^1$-action on $\HH(X/R)$ too.

For a log ring $(R,P)$,
consider the category $\Poly_{(R,P)}$ consisting of the log rings $(R,P)\otimes (\Z[\N^I \oplus \N^J],\N^J)$ for finite sets $I$ and $J$.
Let $\cD$ be an $\infty$-category admitting sifted colimits.
We have an equivalence of $\infty$-categories
\begin{equation}
\Fun_\Sigma(\mathrm{AniLog}_{(R,P)},\cD)
\simeq
\Fun(\Poly_{(R,P)},\cD)
\end{equation}
arguing as in \cite[Construction 2.1]{BMS19},
where $\Fun_\Sigma$ denotes the $\infty$-category of functors preserving sifted colimits,
and $\mathrm{AniLog}_{(R,P)}$ denotes the $\infty$-category of animated log (=pre-log) $(R,P)$-algebras in the sense of \cite[\S 2.2]{BLPO}.
If $F\colon \mathrm{AniLog}_{(R,P)} \to \cD$ corresponds to $f\colon \Poly_{(R,P)}\to \cD$,
we say that $F$ is a \emph{left Kan extension of $f$}.

We have the HKR filtration \cite[Theorem 5.15]{BLPO} on $\HH(X/R)$,
which we review as follows.
For $(A,M)\in \Poly_{(R,\{1\})}$,
the HKR filtration on $\HH((A,M)/R)$ is the Postnikov filtration.
For general log $R$-algebra $(A,M)$,
we get the HKR filtration on $\HH((A,M)/R)$ by left Kan extension.
For general log scheme $X$,
we get the HKR filtration on $\HH(X/R)$ by Zariski descent.

\begin{prop}
\label{HH.1}
Let $(A,M)\to (B,N),(C,L)$ be maps of log $R$-algebras,
where $R$ is a commutative ring.
Then the induced natural $S^1$-equivariant morphism of filtered $\E_\infty$-rings
\[
\HH((B,N)/R)\otimes_{\HH((A,M)/R)}\HH((C,L)/R)
\to
\HH((B,N)\otimes_{(A,M)}^\L (C,L)/R)
\]
is an isomorphism.
\end{prop}
\begin{proof}
Note that the claim is obvious if we are not concerned about the filtrations.

We first treat the special case when $(A,M)=R$.
By left Kan extension,
we may assume $(B,N),(C,L)\in \Poly_{(R,\{1\})}$.
In this case,
the filtrations on $\HH((B,N)/R)$, $\HH((C,L)/R)$, and $\HH((B,N)\otimes_R^\L (C,L)/R)$ are the Postnikov filtrations.
Hence the claim follows from the fact that the tensor product of two $\E_\infty$-rings with the Postnikov filtrations is an $\E_\infty$-ring with the Postnikov filtration.

For general $(A,M)$,
by left Kan extension,
we may assume $(C,L)\in \Poly_{(A,M)}$, i.e.,
$(C,L)=(A,M)\otimes_R (R[\N^I \oplus \N^J],\N^J)$ for some finite sets $I$ and $J$.
Due to the above special case,
we have the natural isomorphisms
\begin{align*}
& \HH((B,N)/R)\otimes_{\HH((A,M)/R)}\HH((C,L)/R)
\\
\simeq &
\HH((B,N)/R)\otimes_{\HH((A,M)/R)}\HH((A,M)/R)\otimes_R \HH((R[\N^I \oplus \N^J],\N^J)/R)
\\
\simeq &
\HH((B,N)/R)\otimes_R \HH((R[\N^I \oplus \N^J],\N^J)/R)
\\
\simeq &
\HH((B,N)\otimes_R (R[\N^I \oplus \N^J],\N^J)/R)
\\
\simeq &
\HH((B,N)\otimes_{(A,M)}^\L (C,L)/R).
\end{align*}
This shows the claim.
\end{proof}

\begin{const}
\label{HH.2}
Let $Z:=\Spec(A/(a_1,\ldots,a_n))\to X:=\Spec(A)$ be a regular immersion of affine schemes,
where $A$ is a commutative ring with a regular sequence $a_1,\ldots,a_n$.
Consider the commutative diagram $\Gys(Z\to X)$ in Remark \ref{cdh.9}.
Using \cite[Lemma 3.5]{MR3748313},
we see that $\Gys(Z\to X)$ is a derived pullback of $\Gys(\Spec(\Z)\xrightarrow{i} \A^n)$,
where $i$ is the inclusion to the origin.
Since derived log schemes are undefined yet but animated log rings are defined,
interpret this result Zariski locally.
\end{const}

\begin{thm}
\label{HH.3}
The presheaf of filtered complexes with $S^1$-action $\HH(-/R)$ on $\lSch/R$ admits logarithmic Gysin sequences.
\end{thm}
\begin{proof}
Since the functor forgetting $S^1$-action is conservative and exact,
it suffices to show that the presheaf of filtered complexes $\HH(-/R)$ admits logarithmic Gysin sequences.
Let $Z\to X$ be a regular immersion of schemes over $R$.
We need to show that the squares in \eqref{intro.1.1} are cartesian for $\HH(-/R)$.
This question is Zariski local on $X$.
Hence  we may assume that $X=\Spec(A)$ for some commutative ring $A$ and $Z=\Spec(A/(a_1,\ldots,a_n))$ for some regular sequence $a_1,\ldots,a_n$ in $A$.
By Proposition \ref{HH.1} and Construction \ref{HH.2},
we reduce to the case when $Z\to X$ is the $0$-section $\Spec(R)\to \A_R^n$.
By Theorem \ref{cdh.7},
it suffices to show that $\HH(-/R)$ is a $(\P^\bullet,\P^{\bullet-1})$-invariant strict Nisnevich sheaf.
This is a consequence of \cite[Proposition 8.3.11, Theorem 8.4.4]{BPO2}.
\end{proof}

\begin{cor}
\label{HH.4}
Let $j$ be an integer.
Then the presheaf of complexes
\[
X\mapsto R\Gamma_{Zar}(X,\L\Omega_{X/R}^j)
\]
on $\lSch/R$ admits logarithmic Gysin sequences.
\end{cor}
\begin{proof}
The $j$th graded piece of $\HH(X/R)$ is $R\Gamma_{Zar}(X,\L\Omega_{X/R}^j)$ by \cite[Theorem 5.15]{BLPO}.
Since the functor $\gr^j$ taking the $j$th graded piece preserves colimits for every integer $j$,
Theorem \ref{HH.3} finishes the proof.
\end{proof}

Recall that $\HC^-(X/R):=\HH(X/R)^{hS^1}$ and $\HP(X/R):=\HH(X/R)^{tS^1}$ for log scheme $X$ over $R$,
where $(-)^{hS^1}$ (resp.\ $(-)^{tS^1}$) denotes the homotopy fixed point (resp.\ Tate construction).

\begin{cor}
\label{HH.5}
The presheaf of complexes $\HC^-(-/R)$ and $\HP(-/R)$ on the category of fs log schemes over $R$ admit logarithmic Gysin sequences.
\end{cor}
\begin{proof}
Let $Z\to X$ be a regular immersion of schemes.
Apply $(-)^{hS^1}$ and $(-)^{tS^1}$ to $\HH(\Gys(Z\to X))$ to conclude.
\end{proof}

\begin{thm}
\label{THH.5}
The presheaf of cyclotomic spectra $\THH$ on $\lSch$ admits logarithmic Gysin sequences.
\end{thm}
\begin{proof}
The forgetful functor from the $\infty$-category of cyclotomic spectra to the $\infty$-category of spectra is conservative and exact by \cite[Corollary II.1.7]{zbMATH07009201}.
Hence it suffices to show that the presheaf of spectra $\THH$ admits logarithmic Gysin sequences.
Argue as in the proof of Theorem \ref{HH.3} to conclude.
\end{proof}

\begin{cor}
\label{THH.6}
The presheaves of complexes $\TC^-$, $\TP$, and $\TC$ on $\lSch$ admit logarithmic Gysin sequences.
\end{cor}
\begin{proof}
Let $Z\to X$ be a regular immersion of schemes.
Apply $(-)^{hS^1}$, $(-)^{tS^1}$, and $\TC$ to $\THH(\Gys(Z\to X))$ to conclude.
\end{proof}

\section{Prismatic cohomology}
\label{THH}

Let $\lQSyn$ denote the category of log quasi-syntomic log rings in \cite[Definition 4.5]{BLPO2},
and let $\lQRSPerfd$ denote the category of log quasi-regular semiperfectoid log rings in  \cite[Definition 4.7]{BLPO2}.
These are the log analogues of the corresponding categories due to Bhatt-Morrow-Scholze in \cite[Definitions 4.10, 4.20]{BMS19}.

A \emph{log quasi-syntomic formal log scheme} is a log formal scheme that is strict \'etale locally isomorphic to $\mathrm{Spf}(A,M)$ for some $(A,M)\in \lQSyn$.
We have the BMS filtrations on $\THH((A,M);\Z_p)$, $\TC^-((A,M);\Z_p)$, $\TP((A,M);\Z_p)$, and $\TC((A,M);\Z_p)$ defined in \cite[\S 7]{BLPO2} and \cite[\S 5]{BPO3}.
For a log quasi-syntomic formal log scheme $X$,
we have the BMS filtrations on $\THH(X;\Z_p)$, $\TC^-(X;\Z_p)$, $\TP(X;\Z_p)$, and $\TC(X;\Z_p)$ by strict \'etale descent.

For a commutative ring $A$ (resp.\ log scheme $X$),
let $A_p^\wedge$ (resp.\ $X_p^\wedge$) denote its $p$-completion.
Then $X_p^\wedge$ is a $p$-adic formal log scheme.
As noted in \cite[Remark 5.3]{BPO3},
we have a natural isomorphism
\[
\THH(X;\Z_p)\simeq
\THH(X_p^\wedge;\Z_p)
\]
of commutative algebra objects of the $\infty$-category of cyclotomic spectra $\CycSp$.
Hence for every log scheme $X$ such that $X_p^\wedge$ is a log quasi-syntomic formal log scheme,
we have the BMS filtration on $\THH(X;\Z_p)$.

The \emph{double speed Postnikov filtration} is $\Fil_{\geq n}:=\tau_{\geq 2n-1}$ for every integer $n$.

\begin{lem}
\label{THH.1}
For $(A,M)\in \lQRSPerfd$,
the BMS filtration on $\THH((A,M)\otimes (\Z[\N],\N);\Z_p)$ is the double speed Postnikov filtration.
\end{lem}
\begin{proof}
Since $\Z\to \Z[\N]$ is smooth and $\Z\to (\Z[\N],\N)$ is log smooth,
there are isomorphisms of complexes
\[
\L_{\Z[\N]/\Z}
\simeq
\Z[\N] 
\simeq
\L_{(\Z[\N],\N)/\Z}.
\]
Using this,
the argument in \cite[Lemma 5.7]{BPO3} can be translated into the setting of $(\Z[\N],\N)$.
\end{proof}

\begin{lem}
\label{THH.8}
For $(A,M)\in \lQRSPerfd$,
there are natural $S^1$-equivariant isomorphisms of filtered $\E_\infty$-rings
\begin{gather*}
\THH((A,M)\otimes \Z[\N];\Z_p)\simeq \THH((A,M);\Z_p)\otimes \HH(\Z[\N]),
\\
\THH((A,M)\otimes (\Z[\N],\N);\Z_p)
\simeq
\THH((A,M);\Z_p)\otimes \HH(\Z[\N],\N).
\end{gather*}
\end{lem}
\begin{proof}
The first isomorphism is in \cite[(5.3)]{BPO3}.
We can similarly obtain the second isomorphism using Lemma \ref{THH.1}.
\end{proof}

\begin{prop}
\label{THH.2}
Let $(R,P)\to (A,M),(B,N)$ be maps of log quasi-syntomic log rings.
Then the induced morphism of filtered $\E_\infty$-rings
\[
\THH((A,M)\otimes^\L_{(R,P)} (B,N);\Z_p)
\to
\THH((A,M);\Z_p)\otimes_{\THH((R,P);\Z_p)}\THH((B,N);\Z_p)
\]
is an isomorphism.
\end{prop}
\begin{proof}
By left Kan extension,
we reduce to the case when $(B,N)=(R,P)\otimes (\Z[\N^{i+j}],\N^j)$ for some integers $i,j\geq 0$.
We proceed by induction on $i$ and $j$.
The claim is clear if $i=j=0$.

If $i=1$ and $j=0$,
then by log quasi-syntomic descent and \cite[Theorem 4.22]{BLPO2},
we reduce to the case when $(A,M),(R,P)\in \lQRSPerfd$.
Apply Lemma \ref{THH.8} to $(A,M)$ and $(R,P)$ to show the desired result.

Assume $i>1$ and $j=0$.
Then we have the natural $S^1$-equivariant isomorphism of filtered $\E_\infty$-rings
\begin{align*}
& \THH((A[x_1,\ldots x_i],M);\Z_p)
\\
\simeq &
\THH((A[ x_1,\ldots,x_{i-1}],M);\Z_p)
\otimes_{\THH(R;\Z_p)}\THH(R[ x_i];\Z_p)
\\
\simeq &
\THH((A,M);\Z_p)\otimes_{\THH(R;\Z_p)}\THH(R[ x_1,\ldots,x_{i-1}];\Z_p)\otimes_{\THH(R;\Z_p)}\THH(R[ x_i];\Z_p)
\\
\simeq &
\THH((A,M);\Z_p)\otimes_{\THH(R;\Z_p)}\THH(R[x_1,\ldots,x_i];\Z_p),
\end{align*}
where the first and third isomorphisms are due to the case of $i=1$ and the second isomorphism is due to the induction hypothesis.

Argue as above for the induction on $j$.
\end{proof}

\begin{lem}
\label{THH.7}
Let $Z\to X$ be a regular immersion of noetherian schemes,
and let $E$ be the exceptional divisor on $\Bl_Z X$.
If $X_p^\wedge$ is quasi-syntomic,
then $(\Bl_Z X)_p^\wedge$ is quasi-syntomic, and $(\Bl_Z X,E)_p^\wedge$ is log quasi-syntomic.
\end{lem}
\begin{proof}
Since the question is strict \'etale local on $X$,
we may assume that $X=\Spec(A)$ for some commutative ring $A$ and $Z=\Spec(A/(a_1,\ldots,a_n))$ for some regular sequence $a_1,\ldots,a_n$ in $A$ such that $A_p^\wedge$ is quasi-syntomic.
Consider the log ring $(B_I,P_I)$ in Example \ref{cdh.4} for every nonempty subset $I$ of $\{1,\ldots,n\}$.

Every $p$-complete noetherian commutative ring has bounded $p^\infty$-torsion,
so it suffices to show that $\L_{(B_I)_p^\wedge /\Z_p}$ and $\L_{((B_I)_p^\wedge,P_I)/\Z_p}$ have $p$-complete Tor amplitude in $[-1,0]$.
Since $A_p^\wedge$ is quasi-syntomic,
$\L_{A_p^\wedge/\Z_p}\otimes^\L_{A_p^\wedge} (B_I)_p^\wedge$ has $p$-complete Tor amplitude in $[-1,0]$ by \cite[Lemma 4.5(1)]{BMS19}.
Due to transitivity sequences for log cotangent complexes,
it suffices to show that $\L_{(B_I)_p^\wedge /A_p^\wedge}$ and $\L_{((B_I)_p^\wedge,P_I)/A_p^\wedge}$ have $p$-complete Tor amplitude in $[-1,0]$.
This question is closed under derived base change,
so we may assume that $Z\to X$ is the $0$-section $\Spec(\Z)\to \A^n$ by Construction \ref{HH.2}.
To finish the proof,
observe that $X$, $\Bl_Z X$, and $(\Bl_Z X,E)$ are log smooth over $\Spec(\Z)$ in this case.
\end{proof}

\begin{thm}
\label{THH.3}
The presheaf of filtered complexes $\THH(-;\Z_p)$ on the category of noetherian fs log schemes whose $p$-completions are log quasi-syntomic admits logarithmic Gysin sequences.
\end{thm}
\begin{proof}
Let $Z\to X$ be a regular immersion of noetherian schemes whose $p$-completions are quasi-syntomic.
By Lemma \ref{THH.7},
we have the diagram \eqref{intro.1.1} for $\THH(-;\Z_p)$ and $Z\to X$.
Argue as in the proof of Theorem \ref{HH.3},
and use Construction \ref{HH.2} and Proposition \ref{THH.2} to reduce to the case when $Z\to X$ is the $0$-section $\Spec(\Z)\to \A^n$.
This smooth case is done by Theorem \ref{cdh.7} and \cite[Proposition 5.8]{BPO3}.
\end{proof}

\begin{cor}
\label{THH.4}
The presheaves of complexes $R\Gamma_{\cPrism}(-)$ and $R\Gamma_{\syn}(-,\Z_p(d))$ for $d\in \Z$ on the category of fs log schemes whose $p$-completions are log quasi-syntomic admit logarithmic Gysin sequences.
\end{cor}
\begin{proof}
We have the Nygaard filtration $\Fil_\rN^{\geq \bullet} R\Gamma_{\cPrism}$ obtained from \cite[\S 7.4]{BLPO2} using strict \'etale descent.
For every integer $j$,
we have the natural isomorphism
\[
\gr_\BMS^j \THH(-;\Z_p)\simeq \gr_\rN^j R\Gamma_{\cPrism}
\]
as observed in the proof of \cite[Proposition 7.4]{BLPO2}.
Hence Theorem \ref{THH.3} implies that $\gr_\rN^j R\Gamma_{\cPrism}(-)$ admits logarithmic Gysin sequences.
It follows that $R\Gamma_{\cPrism}(-)$ admits logarithmic Gysin sequences too.
To show that $R\Gamma_\syn(-,\Z_p(d))$ admits logarithmic Gysin sequences,
use \cite[(5.1)]{BPO3}.
\end{proof}

\bibliography{bib}
\bibliographystyle{siam}

\end{document}